\documentclass[12pt,a4paper]{article}
\usepackage[utf8x]{inputenc}
\usepackage{amsmath}
\usepackage{amsfonts}
\usepackage{amssymb}
\usepackage{amsthm}
\usepackage{mathrsfs}
\usepackage{fancyhdr}
\usepackage{graphicx}
\usepackage[usenames,x11names]{xcolor}
\usepackage{arabtex}
\usepackage{framed}
\usepackage[top=2cm,bottom=2cm,margin=2cm]{geometry}
\usepackage{cancel}
\usepackage{array}  
\usepackage{ulem}

\usepackage{hyperref}

\title{Nontrivial lower bounds for the $p$-adic valuations of some type of rational numbers and an application for establishing the integrality of some rational sequences}
\author{\sc Bakir FARHI \\
National Higher School of Mathematics, \\
P.O.Box 75, Mahelma 16093, Sidi Abdellah (Algiers), Algeria \\[1mm]
\href{mailto:bakir.farhi@nhsm.edu.dz}{bakir.farhi@nhsm.edu.dz} \\[1mm]
\url{http://farhi.bakir.free.fr/}
}
\date{}

\let\up=\textsuperscript

\def\R{{\mathbb R}}

\def\C{{\mathbb C}}
\def\N{{\mathbb N}}
\def\Z{{\mathbb Z}}
\def\lcm{\mathrm{lcm}}
\def\dilog{\mathrm{Li}_2}

\def\restmod#1#2{#1\ (\mathrm{mod}\ #2)} 

\def\idem{\leavevmode\hbox to 10.6mm{\vrule height .63ex depth -.59ex
    width 10mm\hfill}}

\theoremstyle{plain}
\numberwithin{equation}{section}
\newtheorem{thm}{Theorem}[section]

\newtheorem{lemma}[thm]{Lemma}

\newtheorem{coll}[thm]{Corollary}

\theoremstyle{definition}

\theoremstyle{remark}

\newtheorem{rmks}[thm]{Remarks}

\begin{document}
\maketitle

\begin{abstract}
In this note, basing on a certain functional equation of the dilogarithm function, we establish nontrivial lower bounds for the $p$-adic valuation (where $p$ is a given prime number) of some type of rational numbers involving harmonic numbers. Then we use our estimate to derive the integrality of some sequences of rational numbers, which cannot be seen directly from their definitions.         
\end{abstract}


\section{Introduction and Notation}\label{sec1}

Throughout this paper, we let $\N$ denote the set of positive integers and $\N_0 := \N \cup \{0\}$ the set of nonnegative integers. For $x \in \R$, we let $\lfloor x\rfloor$ denote the integer part of $x$. For a given prime number $p$ and a given nonzero rational number $r$, we let $\vartheta_p(r)$ denote the usual $p$-adic valuation of $r$; if in addition $r$ is positive then we let $\log_p(r)$ denote its logarithm to the base $p$ (i.e., $\log_p(r) := \frac{\log{r}}{\log{p}}$). For a given prime number $p$ and a given positive integer $n$, we let $s_p(n)$ denote the sum of base-$p$ digits of $n$. Next, the least common multiple of given positive integers $u_1 , u_2 , \dots , u_n$ ($n \in \N$) is denoted by $\lcm(u_1 , u_2 , \dots , u_n)$. In some places of this paper, we need to use the well-known formulas:
\begin{align}
\vartheta_p\left(\lcm(1 , 2 , \dots , n)\right) & = \left\lfloor\log_p(n)\right\rfloor , \label{eq3} \\[2mm]
\vartheta_p(n!) & = \frac{n - s_p(n)}{p - 1} \label{eq4}
\end{align}
(which are valid for any prime $p$ and any positive integer $n$). Note that only the second one is nontrivial; it is known as \textit{the Legendre formula}, a proof of which can be found in \cite[Theorem 2.6.4, page 77]{moll}). Furthermore, we let ${(H_n)}_{n \in \N_0}$ denote the sequence of harmonic numbers, defined by $H_0 = 0$ and $H_n := \frac{1}{1} + \frac{1}{2} + \dots + \frac{1}{n}$ ($\forall n \in \N$). Finally, we let $\dilog$ denote the dilogarithm function, defined by:
$$
\dilog(X) := \sum_{n = 1}^{+ \infty} \frac{X^n}{n^2} ~~~~~~~~~~ (\forall X \in \C , |X| \leq 1) .
$$
It is known (see e.g., \cite{lew}) that $\dilog$ satisfies the functional equation:
\begin{equation}\label{eq1}
\dilog(X) + \dilog\left(\frac{X}{X - 1}\right) = - \frac{1}{2} \log^2(1 - X)
\end{equation}
(for $X$ in the neighborhood of $0$).

Very recently, the author \cite{far3} have obtained, by different methods, lower bounds for the $p$-adic valuations of the rational numbers of the form $\sum_{k = 1}^{n} \left(\frac{1}{a^k} + \frac{1}{(p - a)^k}\right) \frac{p^k}{k}$ ($n \in \N$, $p$ a prime, $a \in \Z$ with $a \not\equiv \restmod{0}{p}$) (generalizing the earlier results of \cite{far1,far2,dub} which uniquely concern the case $p = 2$). One of these methods exploits the functional equation (analogue to \eqref{eq1}): $\mathrm{Li}_1(X) + \mathrm{Li}_1\left(\frac{X}{X - 1}\right) = 0$ (where $\mathrm{Li}_1(X) := \sum_{n = 1}^{+ \infty} \frac{X^n}{n} = - \log(1 - X)$, for all $X \in \C$ with $|X| < 1$). Following the same method, that we adapt to the function $\dilog$ and its functional equation \eqref{eq1}, we will establish nontrivial lower bounds for the $p$-adic valuations of the rational numbers of the form
$$
\sum_{k = 1}^{n} \left(\frac{1}{a^k} + \frac{1}{(p - a)^k} + \frac{k H_{k - 1}}{a^k}\right) \frac{p^k}{k^2}
$$
($n \in \N$, $p$ a prime, $a \in \Z$ with $a \not\equiv \restmod{0}{p}$). Then by specializing $(p , a)$ to $(2 , 1)$, we derive the integrality of the sequence of general term
$$
\frac{n!^2}{4^n} \sum_{k = 1}^{n} \left(2 + k H_{k - 1}\right) \frac{2^k}{k^2} ~~~~~~~~~~ (n \in \N \setminus \{3 , 5 , 7\})
$$
(a fact which cannot be seen directly from the last expression). We conclude the note by deducing the integrality of another sequence of rational numbers (related to the preceding) and by some general remarks.

\section{The results and the proofs}

Our main result is the following:

\begin{thm}\label{t1}
Let $p$ be a prime number and $a$ be an integer not multiple of $p$. Then we have for all positive integer $n$:
$$
\vartheta_p\left(\sum_{k = 1}^{n} \left(\frac{1}{a^k} + \frac{1}{(p - a)^k} + \frac{k H_{k - 1}}{a^k}\right) \frac{p^k}{k^2}\right) \geq n + 1 - 2 \left\lfloor\log_p(n)\right\rfloor .
$$
\end{thm}

\begin{proof}
By substituting into Equation \eqref{eq1} $X$ by $\frac{X}{a}$, we get
\begin{equation}\label{eq2}
\dilog\left(\frac{X}{a}\right) + \dilog\left(\frac{X}{X - a}\right) + \frac{1}{2} \log^2\left(1 - \frac{X}{a}\right) = 0 .
\end{equation}
Now, let $n \in \N$. Since the $n$\up{th} degree Taylor polynomials of the two functions $t \mapsto \dilog(t)$ and $t \mapsto \frac{1}{2} \log^2(1 - t)$ at $0$ are respectively $\sum_{k = 1}^{n} \frac{t^k}{k^2}$ and $\sum_{k = 1}^{n} \frac{H_{k - 1}}{k} t^k$ and since the functions $t \mapsto \frac{t}{a}$ and $t \mapsto \frac{t}{t - a}$ both vanish at $0$ then (according to the well-known properties of Taylor polynomials) the $n$\up{th} degree Taylor polynomial of the function $X \stackrel{g}{\longmapsto} \dilog\left(\frac{X}{a}\right) + \dilog\left(\frac{X}{X - a}\right) + \frac{1}{2} \log^2\left(1 - \frac{X}{a}\right)$ is the same with the $n$\up{th} degree Taylor polynomial of the rational function
$$
R_n(X) := \sum_{k = 1}^{n} \frac{(\frac{X}{a})^k}{k^2} + \sum_{k = 1}^{n} \frac{(\frac{X}{X - a})^k}{k^2} + \sum_{k = 1}^{n} \frac{H_{k - 1}}{k} \left(\frac{X}{a}\right)^k .
$$
But on the other hand, in view of \eqref{eq2}, the $n$\up{th} degree Taylor polynomial of $g$ at $0$ is zero. Comparing these two results, we deduce that the multiplicity of $0$ in $R_n$ is at least $(n + 1)$. Consequently, $R_n(X)$ can be written as:
$$
R_n(X) = X^{n + 1} \cdot \frac{U_n(X)}{a^n (X - a)^n \lcm(1 , 2 , \dots , n)^2} ,
$$
where $U_n \in \Z[X]$. In particular, we have
$$
R_n(p) = p^{n + 1} \cdot \frac{U_n(p)}{a^n (p - a)^n \lcm(1 , 2 , \dots , n)^2} .
$$
Next, because $U_n(p) \in \Z$ (since $U_n \in \Z[X]$) and $a$ is not a multiple of $p$, then by taking the $p$-adic valuations in the two sides of the last identity, we derive that:
$$
\vartheta_p\left(R_n(p)\right) \geq n + 1 - 2 \vartheta_p\left(\lcm(1 , 2 , \dots , n)\right) = n + 1 - 2 \left\lfloor\log_p(n)\right\rfloor ,
$$
as required. This achieves the proof.
\end{proof}

By taking $(p , a) = (2 , 1)$ in Theorem \ref{t1}, we derive the following important corollary from which we will deduce the integrality of a certain rational sequence.

\begin{coll}\label{coll1}
For all positive integer $n$, we have
\begin{equation}
\vartheta_2\left(\sum_{k = 1}^{n} \left(2 + k H_{k - 1}\right) \frac{2^k}{k^2}\right) \geq n + 1 - 2 \left\lfloor\log_2(n)\right\rfloor . \tag*{$\square$}
\end{equation}
\end{coll} 

As an application of Corollary \ref{coll1}, we obtain the integrality of a particular rational sequence, which cannot be seen directly from its original expression.

\begin{thm}\label{t2}
For every $n \in \N \setminus \{3 , 5 , 7\}$, the rational number
$$
\frac{n!^2}{4^n} \sum_{k = 1}^{n} \left(2 + k H_{k - 1}\right) \frac{2^k}{k^2}
$$
is in fact a positive integer.
\end{thm}

The proof of Theorem \ref{t2} needs the following lemma:

\begin{lemma}\label{l1}
For all integer $n \geq 8$, we have
$$
s_2(n) + \lfloor\log_2(n)\rfloor \leq \frac{n + 1}{2} .
$$
\end{lemma}

\begin{proof}
For $n \in \{8 , 9 , \dots , 16\}$, we verify the required inequality of the lemma by hand. Take for the sequel $n \geq 17$ and let $n = \overline{a_k a_{k - 1} \dots a_0}_{(2)} = a_0 + 2 a_1 + 2^2 a_2 + \dots + 2^k a_k$ be the representation of $n$ in the binary system (with $k \in \N_0$, $a_0 , a_1 , \dots , a_{k - 1} \in \{0 , 1\}$ and $a_k = 1$). Then, we have $s_2(n) = a_0 + a_1 + \dots + a_k \leq k + 1$ and $\lfloor\log_2(n)\rfloor = k$; so $s_2(n) + \lfloor\log_2(n)\rfloor \leq 2 k + 1$. Next, since $n \geq 17$ then $k \geq 4$. We now distinguish the two following cases: \\
\textbullet{} \uline{1\up{st} case:} (If $k = 4$). In this case, we have that:
$$
s_2(n) + \left\lfloor\log_2(n)\right\rfloor \leq 2 k + 1 = 9 = \frac{17 + 1}{2} \leq \frac{n + 1}{2} ,
$$
as required. \\
\textbullet{} \uline{2\up{nd} case:} (If $k \geq 5$). In this case, we have that:
$$
s_2(n) + \left\lfloor\log_2(n)\right\rfloor \leq 2 k + 1 \leq 2^{k - 1} \leq \frac{n}{2} \leq \frac{n + 1}{2} ,
$$
as required. This completes the proof of the lemma.
\end{proof}

\begin{proof}[Proof of Theorem \ref{t2}]
For $n \in \{1 , 2 , 4 , 6\}$, we verify the required result by hand. Take for the sequel $n \geq 8$. Since we have obviously
$$
n!^2 \sum_{k = 1}^{n} \left(2 + k H_{k - 1}\right) \frac{2^k}{k^2} \in \N
$$
then we have just to show that:
$$
\vartheta_2\left(\frac{n!^2}{4^n} \sum_{k = 1}^{n} \left(2 + k H_{k - 1}\right) \frac{2^k}{k^2}\right) \geq 0 .
$$
By using Legendre's formula \eqref{eq4} for $p = 2$ together with Corollary \ref{coll1}, we have that:
\begin{align*}
\vartheta_2\left(\frac{n!^2}{4^n} \sum_{k = 1}^{n} \left(2 + k H_{k - 1}\right) \frac{2^k}{k^2}\right) & = 2 \left(n - s_2(n)\right) - 2 n + \vartheta_2\left(\sum_{k = 1}^{n} \left(2 + k H_{k - 1}\right) \frac{2^k}{k^2}\right) \\
& \geq 2 \left(n - s_2(n)\right) - 2 n + n + 1 - 2 \left\lfloor\log_2(n)\right\rfloor \\
& = n + 1 - 2 \left(s_2(n) + \left\lfloor\log_2(n)\right\rfloor\right) \\
& \geq 0 ~~~~~~~~~~ (\text{according to Lemma \ref{l1}}) ,
\end{align*}
as required. This completes the proof of Theorem \ref{t2}.
\end{proof}

From Theorem \ref{t2}, we derive the following corollary:

\begin{coll}\label{coll2}
For every $n \in \N \setminus \{3 , 5 , 7\}$, the rational number
$$
\frac{(2 n)!^2}{4^n} \sum_{k = 1}^{n} \frac{2 + (n + k) H_{n + k - 1}}{(n + k)^2 2^{n - k}}
$$
is in fact a positive integer.
\end{coll}

\begin{proof}
Let $n \in \N \setminus \{3 , 5 , 7\}$ and set
$$
u_n := \frac{n!^2}{4^n} \sum_{k = 1}^{n} \left(2 + k H_{k - 1}\right) \frac{2^k}{k^2} .
$$
Then we have
\begin{align*}
\frac{(2 n)!^2}{4^n} \sum_{k = 1}^{n} \frac{2 + (n + k) H_{n + k - 1}}{(n + k)^2 2^{n - k}} & = \frac{(2 n)!^2}{4^{2 n}} \sum_{k = 1}^{n} \left(2 + (n + k) H_{n + k - 1}\right) \frac{2^{n + k}}{(n + k)^2} \\[2mm]
& = \frac{(2 n)!^2}{4^{2 n}} \sum_{k = n + 1}^{2 n} \left(2 + k H_{k - 1}\right) \frac{2^k}{k^2} \\[2mm]
& = \frac{(2 n)!^2}{4^{2 n}} \left[\sum_{k = 1}^{2 n} \left(2 + k H_{k - 1}\right) \frac{2^k}{k^2} - \sum_{k = 1}^{n} \left(2 + k H_{k - 1}\right) \frac{2^k}{k^2}\right] \\[2mm]
& = \frac{(2 n)!^2}{4^{2 n}} \sum_{k = 1}^{2 n} \left(2 + k H_{k - 1}\right) \frac{2^k}{k^2} - \frac{(2 n)!^2}{4^n n!^2} \cdot \frac{n!^2}{4^n} \sum_{k = 1}^{n} \left(2 + k H_{k - 1}\right) \frac{2^k}{k^2} \\[2mm]
& = u_{2 n} - \left(\frac{(2 n)!}{2^n n!}\right)^2 u_n .
\end{align*}
But since $u_n , u_{2 n} \in \Z$ (according to Theorem \ref{t2}) and $\frac{(2 n)!}{2^n n!} = \frac{1 \times 2 \times \cdots \times (2 n)}{2 \times 4 \times \cdots \times (2 n)} = 1 \times 3 \times 5 \times \cdots \times (2 n - 1) \in \Z$, we have that $u_{2 n} - \left(\frac{(2 n)!}{2^n n!}\right)^2 u_n \in \Z$. The required result of the corollary then follows.
\end{proof}

\begin{rmks}~
\begin{enumerate}
\item Using Theorem 2.2 of \cite{far3}, we can easily verify that the lower bound of Theorem \ref{t1} is essentially optimal.
\item The rational sequence introduced in Theorem \ref{t2} can be alternatively defined by the recurrence:
$$
\left\{\begin{array}{l}
u_0 = 0 \\
u_n = \frac{n^2}{4} u_{n - 1} + \frac{(n - 1)!^2}{2^n} \left(2 + n H_{n - 1}\right) ~~~~ (\forall n \geq 1)
\end{array}
\right. .
$$ 
Similarly, the rational sequence introduced in Corollary \ref{coll2} can be alternatively defined by the recurrence:
$$
\left\{\begin{array}{l}
v_0 = 0 \\
v_n = \frac{n^2 (2 n - 1)^2}{4} v_{n - 1} + K_n ~~~~ (\forall n \geq 1)
\end{array}
\right. ,
$$
where ${(K_n)}_{n}$ is a sequence of rational numbers having a closed form in terms of harmonic numbers. 
\item Using some more complicated functional equations of polylogarithms (such as Equation (6.108) of \cite[page 178]{lew}), we can give other results similar to Theorem \ref{t1} and then establish the integrality of some other rational sequences similar to those of Theorem \ref{t2} and Corollary \ref{coll2}.
\item It is also possible to prove Theorem \ref{t1} by means of the $p$-adic dilogarithm function together with Theorem 2.2 of \cite{far3} (this non elementary method is detailed in \cite{far3} for the $p$-adic logarithm function).
\end{enumerate}
\end{rmks}

\rhead{\textcolor{OrangeRed3}{\it References}}

\end{document}